\documentclass{amsart}
\usepackage{amssymb}
\usepackage{amsfonts}

\newtheorem{theorem}{Theorem}
\theoremstyle{plain}

\numberwithin{equation}{section}
\input{tcilatex}

\begin{document}
\title[New Ostrowski-Type Inequality]{A New Ostrowski-Type Inequality for
Double Integrals}
\author{$^{\blacktriangledown }$M. Emin \"{O}zdemir}
\address{$^{\blacktriangledown }$Ataturk University, K. K. Education
Faculty, Department of Mathematics, 25640, Kampus, Erzurum, Turkey}
\email{emos@atauni.edu.tr}
\author{$^{\spadesuit ,\bigstar }$Ahmet Ocak Akdemir}
\curraddr{$^{\bigstar }$A\u{g}r\i\ \.{I}brahim \c{C}e\c{c}en University,
Faculty of Science and Arts, Department of Mathematics, 04100, A\u{g}r\i ,
Turkey}
\email{ahmetakdemir@agri.edu.tr}
\author{$^{\blacksquare }$Erhan Set}
\address{$^{\blacksquare }$Ataturk University, K. K. Education Faculty,
Department of Mathematics, 25640, Kampus, Erzurum, Turkey}
\date{October 20, 2010}
\keywords{Ostrowski Inequality\\
$^{\spadesuit }Corresponding$ $Author$}

\begin{abstract}
In this paper, we established a new Ostrowski-type inequality involving
functions of two independent variables.
\end{abstract}

\maketitle

\section{INTRODUCTION}

In \cite{OST}, Ostrowski proved the following inequality.

\begin{theorem}
Let $f:I\rightarrow 
\mathbb{R}
,$ where $I\subset 
\mathbb{R}
$ is an interval, be a mapping differentiable in the interior of $I$ and $%
a,b\in I^{o},a<b.$ If $\left\vert f^{\prime }\right\vert \leq M,\forall t\in
\lbrack a,b],$ then we have%
\begin{equation}
\left\vert f(x)-\frac{1}{b-a}\dint\limits_{a}^{b}f(t)dt\right\vert \leq %
\left[ \frac{1}{4}+\frac{\left( x-\frac{a+b}{2}\right) ^{2}}{\left(
b-a\right) ^{2}}\right] \left( b-a\right) M,  \label{1.1}
\end{equation}%
for $x\in \lbrack a,b].$
\end{theorem}

In \cite{ME}, \"{O}zdemir et al. proved inequalities as above for  $\left(
\alpha ,m\right) -$convex functions. In \cite{CHENG}, Cheng proved the
following inequality,

\begin{theorem}
Let $I\subset 
\mathbb{R}
$ be an open interval, $a,b\in I,a<b.$ $f:I\rightarrow 
\mathbb{R}
$ is a differentiable function such that there exist constants $\gamma
,\Gamma \in 
\mathbb{R}
$ with $\gamma \leq f^{\prime }(x)\leq \Gamma ,x\in \lbrack a,b].$ Then we
have%
\begin{eqnarray*}
&&\left\vert \frac{1}{2}f(x)-\frac{\left( x-b\right) f\left( b\right)
-\left( x-a\right) f\left( a\right) }{2\left( b-a\right) }-\frac{1}{b-a}%
\dint\limits_{a}^{b}f(t)dt\right\vert \\
&\leq &\frac{\left( x-a\right) ^{2}+\left( b-x\right) ^{2}}{8\left(
b-a\right) }(\Gamma -\gamma )
\end{eqnarray*}%
for all $x\in \lbrack a,b].$
\end{theorem}

Similarly, in \cite{UJE}, Ujevic established some double integral
inequalities and in \cite{LIU}, Liu et.al. proved two sharp inequalities of
perturbed Ostrowski-type. Recently, in \cite{ZEKI}, Sar\i kaya established
following integral inequality of Ostrowski-type involving functions of two
independent variables;

\begin{theorem}
Let $f:\left[ a,b\right] \times \left[ c,d\right] \rightarrow 
\mathbb{R}
$ be an absolutely continuous function such that the partial derivative of
order 2 exists and supposes that there exist constants $\gamma ,\Gamma \in 
\mathbb{R}
$ with $\gamma \leq \frac{\partial ^{2}f(t,s)}{\partial t\partial s}\leq
\Gamma $ for all $(t,s)\in \left[ a,b\right] \times \left[ c,d\right] .$
Then, we have%
\begin{eqnarray*}
&&\left\vert \frac{1}{4}f(x,y)+\frac{1}{4}H(x,y)-\frac{1}{2(b-a)}%
\dint\limits_{a}^{b}f(t,y)dt-\frac{1}{2(d-c)}\dint\limits_{c}^{d}f(x,s)ds%
\right. \\
&&-\frac{1}{2(b-a)(d-c)}\dint\limits_{a}^{b}\left[ (y-c)f(t,c)+(d-y)f(t,d)%
\right] dt \\
&&\left. -\frac{1}{2(b-a)(d-c)}\dint\limits_{c}^{d}\left[
(x-a)f(a,s)+(b-x)f(b,s)\right] ds+\frac{1}{2(b-a)(d-c)}\dint\limits_{a}^{b}%
\dint\limits_{c}^{d}f(t,s)dsdt\right\vert \\
&\leq &\frac{\left[ (x-a)^{2}+(b-x)^{2}\right] \left[ (y-c)^{2}+(d-y)^{2}%
\right] }{32(b-a)(d-c)}(\Gamma -\gamma )
\end{eqnarray*}%
for all $(x,y)\in \left[ a,b\right] \times \left[ c,d\right] $ where%
\begin{eqnarray*}
&&H(x,y) \\
&=&\dfrac{(x-a)\left[ (y-c)f(a,c)+(d-y)f(a,d)\right] +(b-x)\left[
(y-c)f(b,c)+(d-y)f(b,d)\right] }{(b-a)(d-c)} \\
&&+\frac{(x-a)f(a,y)+(b-x)f(b,y)}{b-a}+\frac{(y-c)f(x,c)+(d-y)f(x,d)}{d-c}
\end{eqnarray*}
\end{theorem}

In \cite{QIO}, Qiaoling et.al. derived a new inequality of Ostrowski-type as
following

\begin{theorem}
Let $f:\left[ a,b\right] \times \left[ c,d\right] \rightarrow 
\mathbb{R}
$ be an absolutely continuous function such that the partial derivative of
order 2 exists and suppose that there exist constants $\gamma ,\Gamma \in 
\mathbb{R}
$ with $\gamma \leq \frac{\partial ^{2}f(t,s)}{\partial t\partial s}\leq
\Gamma $ for all $(t,s)\in \left[ a,b\right] \times \left[ c,d\right] .$
Then, we have%
\begin{eqnarray*}
&&\left\vert (1-\lambda )^{2}f(x,y)+\frac{\lambda }{2}(1-\lambda )\left[
f(a,y)+f(b,y)+f(x,c)+f(x,d)\right] \right. \\
&&+\left( \frac{\lambda }{2}\right) ^{2}\left[ f(a,c)+f(b,c)+f(a,d)+f(b,d)%
\right] \\
&&-\frac{1}{b-a}\left\{ (1-\lambda )\dint\limits_{a}^{b}f(t,y)dt+\frac{%
\lambda }{2}\dint\limits_{a}^{b}\left[ f(t,c)+f(t,d)\right] dt\right\} \\
&&-\frac{1}{d-c}\left\{ (1-\lambda )\dint\limits_{c}^{d}f(x,s)ds+\frac{%
\lambda }{2}\dint\limits_{c}^{d}\left[ f(a,s)+f(b,s)\right] ds\right\} \\
&&\left. -\frac{\Gamma +\gamma }{2}(1-\lambda )^{2}\left( x-\frac{a+b}{2}%
\right) \left( y-\frac{c+d}{2}\right) +\frac{1}{(b-a)(d-c)}%
\dint\limits_{a}^{b}\dint\limits_{c}^{d}f(t,s)dsdt\right\vert \\
&\leq &\frac{\Gamma -\gamma }{2}\frac{1}{(b-a)(d-c)}\left[ \left( \lambda
^{2}+(1-\lambda )^{2}\right) \frac{(b-a)^{2}}{4}+\left( x-\frac{a+b}{2}%
\right) ^{2}\right] \\
&&\times \left[ \left( \lambda ^{2}+(1-\lambda )^{2}\right) \frac{(d-c)^{2}}{%
4}+\left( y-\frac{c+d}{2}\right) ^{2}\right]
\end{eqnarray*}%
for all $(x,y)\in \left[ a+\lambda \frac{b-a}{2},b-\lambda \frac{b-a}{2}%
\right] \times \left[ c+\lambda \frac{d-c}{2},d-\lambda \frac{d-c}{2}\right] 
$ and $\lambda \in \lbrack 0,1].$
\end{theorem}

In this paper, we proved a new Ostrowski-type inequality involving functions
of two independent variables as above.

\section{MAIN\ RESULT}

\begin{theorem}
Let $f:\left[ a,b\right] \times \left[ c,d\right] \rightarrow 
\mathbb{R}
$ be an absolutely continuous function such that the partial derivative of
order 2 exists and supposes that there exist constants $\gamma ,\Gamma \in 
\mathbb{R}
$ with $\gamma \leq \frac{\partial ^{2}f(t,s)}{\partial t\partial s}\leq
\Gamma $ for all $(t,s)\in \left[ a,b\right] \times \left[ c,d\right] .$
Then, we have%
\begin{eqnarray}
&&\left\vert \frac{1}{16}Kf(x,y)+\frac{1}{16}H(x,y)\right.  \notag \\
&&-\frac{1}{4\left( b-a\right) \left( d-c\right) }\dint\limits_{c}^{d}\left[
3(x-a)f(x,s)-(b-x)f(x,s)\right] ds  \notag \\
&&-\frac{1}{4\left( b-a\right) \left( d-c\right) }\dint\limits_{a}^{b}\left[
3(y-c)f(t,y)-(d-y)f(t,y)\right] dt  \notag \\
&&-\frac{1}{4\left( b-a\right) \left( d-c\right) }\dint\limits_{c}^{d}\left[
3(b-x)f(b,s)-(x-a)f(a,s)\right] ds  \label{2.1} \\
&&-\frac{1}{4\left( b-a\right) \left( d-c\right) }\dint\limits_{a}^{b}\left[
3(d-y)f(t,d)-(y-c)f(t,c)\right] dt  \notag \\
&&-\frac{\left[ \left( y-c\right) ^{2}-\left( d-y\right) ^{2}\right] \left[
\left( x-a\right) ^{2}-\left( b-x\right) ^{2}\right] }{32\left( b-a\right)
\left( d-c\right) }\left( \Gamma +\gamma \right)  \notag \\
&&\left. +\frac{1}{\left( b-a\right) \left( d-c\right) }\dint\limits_{a}^{b}%
\dint\limits_{c}^{d}f(t,s)dsdt\right\vert  \notag \\
&\leq &\frac{25\left[ \left( y-c\right) ^{2}+\left( d-y\right) ^{2}\right] %
\left[ \left( x-a\right) ^{2}+\left( b-x\right) ^{2}\right] }{512\left(
b-a\right) \left( d-c\right) }\left( \Gamma -\gamma \right)  \notag
\end{eqnarray}%
for all $(x,y)\in \left[ a,b\right] \times \left[ c,d\right] ,$ where%
\begin{eqnarray*}
H(x,y) &=&\frac{\left[ 3(b-x)f(b,y)-(x-a)f(a,y)\right] \left(
3(y-c)-(d-y)\right) }{(b-a)(d-c)} \\
&&+\frac{\left[ 3(d-y)f(x,d)-(y-c)f(x,c)\right] \left( 3(x-a)-(b-x)\right) }{%
(b-a)(d-c)} \\
&&+\frac{\left[ (y-c)f(a,c)-3(d-y)f(a,d)\right] (x-a)}{(b-a)(d-c)} \\
&&+\frac{\left[ 3(d-y)f(b,d)-(y-c)f(b,c)\right] (b-x)}{(b-a)(d-c)}
\end{eqnarray*}%
and%
\begin{equation*}
K=\frac{\left[ \left( 3(x-a)-(b-x)\right) \left( 3(y-c)-(d-y)\right) \right] 
}{\left( b-a\right) \left( d-c\right) }.
\end{equation*}
\end{theorem}

\begin{proof}
We define the functions: $p:\left[ a,b\right] ^{2}\rightarrow 
\mathbb{R}
$ and $q:[c,d]^{2}\rightarrow 
\mathbb{R}
$ as following%
\begin{equation*}
p(x,t)=\left\{ 
\begin{array}{c}
t-\frac{3a+x}{4}\text{ \ \ \ },t\in \lbrack a,x] \\ 
\\ 
t-\frac{3b+x}{4}\text{ \ \ \ },t\in (x,b]%
\end{array}%
\right.
\end{equation*}%
and%
\begin{equation*}
q(y,s)=\left\{ 
\begin{array}{c}
s-\frac{3c+y}{4}\text{ \ \ \ },s\in \lbrack c,y] \\ 
\\ 
s-\frac{3d+y}{4}\text{ \ \ \ },s\in (y,d]%
\end{array}%
\right.
\end{equation*}%
From definitions of $\ p(x,t)$ and $q(y,s)$, we can write%
\begin{eqnarray}
&&\dint\limits_{a}^{b}\dint\limits_{c}^{d}p(x,t)q(y,s)\frac{\partial
^{2}f(t,s)}{\partial t\partial s}dsdt  \label{2.2} \\
&=&\dint\limits_{a}^{x}\dint\limits_{c}^{y}\left( t-\frac{3a+x}{4}\right)
\left( s-\frac{3c+y}{4}\right) \frac{\partial ^{2}f(t,s)}{\partial t\partial
s}dsdt  \notag \\
&&+\dint\limits_{a}^{x}\dint\limits_{y}^{d}\left( t-\frac{3a+x}{4}\right)
\left( s-\frac{3d+y}{4}\right) \frac{\partial ^{2}f(t,s)}{\partial t\partial
s}dsdt  \notag \\
&&+\dint\limits_{x}^{b}\dint\limits_{c}^{y}\left( t-\frac{3b+x}{4}\right)
\left( s-\frac{3c+y}{4}\right) \frac{\partial ^{2}f(t,s)}{\partial t\partial
s}dsdt  \notag \\
&&+\dint\limits_{x}^{b}\dint\limits_{y}^{d}\left( t-\frac{3b+x}{4}\right)
\left( s-\frac{3d+y}{4}\right) \frac{\partial ^{2}f(t,s)}{\partial t\partial
s}dsdt  \notag
\end{eqnarray}%
Computing each integral of right hand side of (\ref{2.2}), we have%
\begin{eqnarray}
&&\dint\limits_{a}^{x}\dint\limits_{c}^{y}\left( t-\frac{3a+x}{4}\right)
\left( s-\frac{3c+y}{4}\right) \frac{\partial ^{2}f(t,s)}{\partial t\partial
s}dsdt  \label{2.3} \\
&=&\frac{(x-a)(y-c)}{16}\left[ 9f(x,y)-3f(x,c)-3f(a,y)+f(a,c)\right]  \notag
\\
&&-\frac{(x-a)}{4}\dint\limits_{c}^{y}\left[ 3f(x,s)-f(a,s)\right] ds-\frac{%
(y-c)}{4}\dint\limits_{a}^{x}\left[ 3f(t,y)-f(t,c)\right] dt  \notag \\
&&+\dint\limits_{a}^{x}\dint\limits_{c}^{y}f(t,s)dsdt  \notag
\end{eqnarray}%
\begin{eqnarray}
&&\dint\limits_{a}^{x}\dint\limits_{y}^{d}\left( t-\frac{3a+x}{4}\right)
\left( s-\frac{3d+y}{4}\right) \frac{\partial ^{2}f(t,s)}{\partial t\partial
s}dsdt  \label{2.4} \\
&=&\frac{(x-a)(d-y)}{16}\left[ 9f(x,d)-3f(x,y)-3f(a,d)+f(a,y)\right]  \notag
\\
&&-\frac{(x-a)}{4}\dint\limits_{y}^{d}\left[ 3f(x,s)-f(a,s)\right] ds-\frac{%
(d-y)}{4}\dint\limits_{a}^{x}\left[ 3f(t,d)-f(t,y)\right] dt  \notag \\
&&+\dint\limits_{a}^{x}\dint\limits_{y}^{d}f(t,s)dsdt  \notag
\end{eqnarray}%
\begin{eqnarray}
&&\dint\limits_{x}^{b}\dint\limits_{c}^{y}\left( t-\frac{3b+x}{4}\right)
\left( s-\frac{3c+y}{4}\right) \frac{\partial ^{2}f(t,s)}{\partial t\partial
s}dsdt  \label{2.5} \\
&=&\frac{(b-x)(y-c)}{16}\left[ 9f(b,y)-3f(b,c)-3f(x,y)+f(x,c)\right]  \notag
\\
&&-\frac{(b-x)}{4}\dint\limits_{c}^{y}\left[ 3f(b,s)-f(x,s)\right] ds-\frac{%
(y-c)}{4}\dint\limits_{x}^{b}\left[ 3f(t,y)-f(t,c)\right] dt  \notag \\
&&+\dint\limits_{x}^{b}\dint\limits_{c}^{y}f(t,s)dsdt  \notag
\end{eqnarray}%
\begin{eqnarray}
&&\dint\limits_{x}^{b}\dint\limits_{y}^{d}\left( t-\frac{3b+x}{4}\right)
\left( s-\frac{3d+y}{4}\right) \frac{\partial ^{2}f(t,s)}{\partial t\partial
s}dsdt  \label{2.6} \\
&=&\frac{(b-x)(d-y)}{16}\left[ 9f(b,d)-3f(b,y)-3f(x,d)+f(x,y)\right]  \notag
\\
&&-\frac{(b-x)}{4}\dint\limits_{y}^{d}\left[ 3f(b,s)-f(x,s)\right] ds-\frac{%
(d-y)}{4}\dint\limits_{x}^{b}\left[ 3f(t,d)-f(t,y)\right] dt  \notag \\
&&+\dint\limits_{x}^{b}\dint\limits_{y}^{d}f(t,s)dsdt  \notag
\end{eqnarray}%
By using these inequalities in (\ref{2.2}), we get%
\begin{eqnarray}
&&\dint\limits_{a}^{b}\dint\limits_{c}^{d}p\left( x,t\right) q\left(
y,s\right) \frac{\partial ^{2}f(t,s)}{\partial t\partial s}dsdt  \notag \\
&=&\frac{1}{16}\left\{ \left[ \left( 3(x-a)-(b-x)\right) \left(
3(y-c)-(d-y)\right) \right] f(x,y)\right.  \notag \\
&&+\left[ 3(b-x)f(b,y)-(x-a)f(a,y)\right] \left( 3(y-c)-(d-y)\right)  \notag
\\
&&+\left[ 3(d-y)f(x,d)-(y-c)f(x,c)\right] \left( 3(x-a)-(b-x)\right)  \notag
\\
&&+\left[ (y-c)f(a,c)-3(d-y)f(a,d)\right] (x-a)  \notag \\
&&\left. +\left[ 3(d-y)f(b,d)-(y-c)f(b,c)\right] (b-x)\right\}  \label{2.7}
\\
&&-\frac{1}{4}\dint\limits_{c}^{d}\left[ 3(x-a)f(x,s)-(b-x)f(x,s)\right] ds 
\notag \\
&&-\frac{1}{4}\dint\limits_{a}^{b}\left[ 3(y-c)f(t,y)-(d-y)f(t,y)\right] dt 
\notag \\
&&-\frac{1}{4}\dint\limits_{c}^{d}\left[ 3(b-x)f(b,s)-(x-a)f(a,s)\right] ds 
\notag \\
&&-\frac{1}{4}\dint\limits_{a}^{b}\left[ 3(d-y)f(t,d)-(y-c)f(t,c)\right] dt 
\notag \\
&&+\dint\limits_{a}^{b}\dint\limits_{c}^{d}f(t,s)dsdt  \notag
\end{eqnarray}%
We\ also\ have%
\begin{equation}
\dint\limits_{a}^{b}\dint\limits_{c}^{d}p\left( x,t\right) q\left(
y,s\right) dsdt=\frac{\left[ \left( y-c\right) ^{2}-\left( d-y\right) ^{2}%
\right] \left[ \left( x-a\right) ^{2}-\left( b-x\right) ^{2}\right] }{16}
\label{2.8}
\end{equation}%
Let $M=\frac{\Gamma +\gamma }{2}.$ From (\ref{2.7}) and (\ref{2.8}), we can
write%
\begin{eqnarray}
&&\dint\limits_{a}^{b}\dint\limits_{c}^{d}p\left( x,t\right) q\left(
y,s\right) \left[ \frac{\partial ^{2}f(t,s)}{\partial t\partial s}-M\right]
dsdt  \label{2.9} \\
&=&\dint\limits_{a}^{b}\dint\limits_{c}^{d}p\left( x,t\right) q\left(
y,s\right) \frac{\partial ^{2}f(t,s)}{\partial t\partial s}dsdt  \notag \\
&&-\frac{\Gamma +\gamma }{2}\frac{\left[ \left( y-c\right) ^{2}-\left(
d-y\right) ^{2}\right] \left[ \left( x-a\right) ^{2}-\left( b-x\right) ^{2}%
\right] }{16}  \notag
\end{eqnarray}%
On the other hand, we have%
\begin{eqnarray}
&&\left\vert \dint\limits_{a}^{b}\dint\limits_{c}^{d}p\left( x,t\right)
q\left( y,s\right) \left[ \frac{\partial ^{2}f(t,s)}{\partial t\partial s}-M%
\right] dsdt\right\vert  \label{2.10} \\
&\leq &\max_{(t,s)\in \lbrack a,b]\times \lbrack c,d]}\left\vert \frac{%
\partial ^{2}f(t,s)}{\partial t\partial s}-M\right\vert
\dint\limits_{a}^{b}\dint\limits_{c}^{d}\left\vert p\left( x,t\right)
q\left( y,s\right) \right\vert dsdt  \notag
\end{eqnarray}%
Since%
\begin{equation}
\max_{(t,s)\in \lbrack a,b]\times \lbrack c,d]}\left\vert \frac{\partial
^{2}f(t,s)}{\partial t\partial s}-M\right\vert \leq \frac{\Gamma -\gamma }{2}
\label{2.11}
\end{equation}%
and%
\begin{equation}
\dint\limits_{a}^{b}\dint\limits_{c}^{d}\left\vert p\left( x,t\right)
q\left( y,s\right) \right\vert dsdt=\frac{25\left[ \left( y-c\right)
^{2}+\left( d-y\right) ^{2}\right] \left[ \left( x-a\right) ^{2}+\left(
b-x\right) ^{2}\right] }{256}  \label{2.12}
\end{equation}%
By using (\ref{2.11}) and (\ref{2.12}) in (\ref{2.10}), we get%
\begin{eqnarray}
&&\left\vert \dint\limits_{a}^{b}\dint\limits_{c}^{d}p\left( x,t\right)
q\left( y,s\right) \left[ \frac{\partial ^{2}f(t,s)}{\partial t\partial s}-M%
\right] dsdt\right\vert  \label{2.13} \\
&\leq &\frac{25\left[ \left( y-c\right) ^{2}+\left( d-y\right) ^{2}\right] %
\left[ \left( x-a\right) ^{2}+\left( b-x\right) ^{2}\right] }{512}\left(
\Gamma -\gamma \right)  \notag
\end{eqnarray}%
From (\ref{2.9}) and (\ref{2.13}), we get the required result.
\end{proof}

\end{document}